\documentclass[12pt]{article}

\usepackage{amsmath, amsthm, amssymb}
\usepackage{mathtools}
\usepackage{indentfirst}

\theoremstyle{plain}
\newtheorem{theorem}{Theorem}
\newtheorem{lemma}[theorem]{Lemma}

\newcommand{\arxiv}[1]{\href{http://arxiv.org/abs/#1}{\texttt{arXiv:#1}}}

\DeclarePairedDelimiter\abs{\lvert}{\rvert}

\title{Convergence of a sinusoidal infinite series\\
       from Borwein, Bailey, and Girgensohn}

\author{Ravi B. Boppana\thanks{Department of Mathematics, Massachusetts Institute of Technology, Cambridge, Massachusetts, USA\@.  Email address: {\tt rboppana@mit.edu}.}
}

\date{July 20, 2020}

\usepackage{hyperref}
\hypersetup{colorlinks}

\begin{document}

\maketitle

\begin{abstract}
Borwein, Bailey, and Girgensohn (2004) asked whether the following infinite series converges: 
the sum of $(\frac{2}{3} + \frac{1}{3} \sin n)^n / n$ over all positive integers $n$.
We prove that their series converges.
The proof uses the irrationality measure of $\pi$.  
\end{abstract}  

\section{Introduction}

In their book on experimentation in mathematics, 
Borwein, Bailey, and Girgensohn~\cite[page~56]{BBG} asked whether the infinite series
\[
  \sum_{n=1}^\infty \frac{(\frac{2}{3} + \frac{1}{3} \sin n)^n}{n}
\]
converges.  
They wrote that this is an open problem.  
Here we prove that their series converges.\footnote{We originally solved this problem in 2005 on the Art of Problem Solving forum~\cite{AoPS}.}  

Because each term of the series is positive, 
convergence of the series is equivalent to the partial sums being bounded above by a constant.
The $n$th term of the series is bounded above by~$\frac{1}{n}$,
but that bound by itself does not settle convergence,
because the harmonic series $\sum_n \frac{1}{n}$ diverges.  
Instead, we will check that most terms are quite small, 
and will carefully handle the exceptional terms. 
To cope with the exceptional terms, 
we use Mahler's theorem~\cite{Mahler} that $\pi$ isn't too close to a rational number.   

\section{Proof of convergence}

Let $n$ be a positive integer.  
Say that $n$ is \emph{tame} if for every integer~$a$,
\[
  \abs[\Big]{ n - \frac{\pi}{2} - 2 \pi a } \geq \frac{4}{n^{1/4}} \, .
\]
Roughly speaking, $n$ being tame means that $\sin n$ is not too close to~$1$.
Say that $n$ is \emph{wild} if it is not tame.

For our infinite series, we show that the terms with tame indexes are tiny, 
and so their sum is bounded.  
On the other hand, we show that wild indexes are rare, 
and so the sum of their terms is also bounded.  

Our first lemma shows that, as promised, the terms with tame indexes are tiny.  

\begin{lemma} \label{lemma:tame}
If $n$ is a tame positive integer, then 
\[
  \Bigl( \frac{2}{3} + \frac{1}{3} \sin n \Bigr)^n \le e^{- \sqrt{n}} . 
\]
\end{lemma}

\begin{proof} 
Let $a$ be an integer such that $\abs{ n - \frac{\pi}{2} - 2 \pi a } \le \pi$.
Let $\theta$ be $n - \frac{\pi}{2} - 2 \pi a$.
By the double-angle identity for cosine, 
\[
  1 - \sin n 
	  = 1 - \cos \theta 
		= 2 \sin^2 \frac{\theta}{2} \, .
\]
Because $\abs{\sin x} \ge \frac{2}{\pi} \abs{x}$ for $\abs{x} \le \frac{\pi}{2}$, we have
\[
  1 - \sin n 
	  = 2 \sin^2 \frac{\theta}{2} 
	  \ge 2 \Bigl( \frac{2}{\pi} \cdot \frac{\theta}{2} \Bigr)^2 
		= \frac{2}{\pi^2} \theta^2 .
\]
Because $n$ is tame, we have $\abs{\theta} \ge 4 / n^{1/4}$, and so
\[
  1 - \sin n 
		\ge \frac{2}{\pi^2} \theta^2 
		\ge \frac{2}{\pi^2} \Bigl( \frac{4}{n^{1/4}} \Bigr)^2
		= \frac{32}{\pi^2 \sqrt{n}}
		\ge \frac{3}{\sqrt{n}} \, .
\]
Hence 
\[
  \frac{2}{3} + \frac{1}{3} \sin n
	  \le \frac{2}{3} + \frac{1}{3} \Bigl( 1 - \frac{3}{\sqrt{n}} \Bigr)
		= 1 - \frac{1}{\sqrt{n}} \, .
\]
Taking the $n$th power gives the desired bound:
\[
  \Bigl( \frac{2}{3} + \frac{1}{3} \sin n \Bigr)^n
	  \le \Bigl( 1 - \frac{1}{\sqrt{n}} \Bigr)^n
		\le e^{- n / \sqrt{n}}
		= e^{- \sqrt{n}} .  \qedhere
\]
\end{proof}

To deal with the wild numbers, 
we use a known result that $\pi$ is not too close to a rational number.
Mahler~\cite{Mahler} first proved the theorem below, 
except with the weaker exponent of~$42$ in place of the~$20$.  
The version below is due to Chudnovsky~\cite{Chudnovsky}.
The latest result on the ``irrationality measure'' of~$\pi$ is by Zeilberger and Zudilin~\cite{ZZ},
who improved the exponent to~$7.2$,
provided that $\abs{q}$ is sufficiently large.

\begin{theorem} \label{thm:Mahler}
If $p$ and $q$ are integers such that $\abs{q} > 1$, then
\[
  \abs[\Big]{ \pi - \frac{p}{q} } > \frac{1}{\abs{q}^{20}} \, .
\]
\end{theorem}

We now handle the wild numbers.
Given a positive integer~$k$, 
let $W_k$ be the $k$th smallest wild number.
Just in case,
if there are fewer than $k$ wild numbers,
then let $W_k$ be $+ \infty$.

Our next lemma shows that the wild numbers grow pretty large.  

\begin{lemma} \label{lemma:wild}
If $k$ is a positive integer, 
then the $k$th wild number~$W_k$ is at least~$\frac{1}{2} k^{77/76}$.
\end{lemma}

\begin{proof}
The proof is by induction on~$k$.
The base cases $k \le 2^{76}$ are simple:
\[
  W_k 
	  \ge k
		= 2 \cdot \frac{1}{2} k
		\ge k^{1/76} \cdot \frac{1}{2} k
		= \frac{1}{2} k^{77/76} .
\]

Given $k \ge 2^{76}$,
the induction hypothesis is that $W_k \ge \frac{1}{2} k^{77/76}$.
We will prove that $W_{k + 1} \ge \frac{1}{2} (k + 1)^{77/76}$.
For convenience, let $x = W_k$ and $y = W_{k + 1}$.
If $y = +\infty$, then we are done, so we may assume that $x$ and $y$ are actual integers.  
We will show that $y$ is far from~$x$.  

Because $x$ is wild, there is an integer~$a$ such that 
\[
  \abs[\Big]{ x - \frac{\pi}{2} - 2 \pi a } < \frac{4}{x^{1/4}} \, .
\]
Similarly, because $y$ is wild, there is an integer~$b$ such that 
\[
  \abs[\Big]{ y - \frac{\pi}{2} - 2 \pi b } 
	  < \frac{4}{y^{1/4}} \, .
\]
Subtracting and using the triangle inequality gives
\[
  \abs{y - x - 2 (b - a) \pi} 
	  < \frac{4}{x^{1/4}} + \frac{4}{y^{1/4}}
		< \frac{8}{x^{1/4}} \, .
\]
Let $p = y - x$ and $q = 2(b - a)$. 
Then the inequality becomes
\[
  \abs{p - q \pi} < \frac{8}{x^{1/4}} \, .
\]
Because $x = W_k \ge k \ge 2^{76}$,
the right side $8 / x^{1/4}$ is less than~$1$.  
Hence, because $p \ge 1$, we have $q > 0$.
Because $q$ is even, it follows that $q \ge 2$.      
By Theorem~\ref{thm:Mahler} (Mahler's theorem), 
\[
  \frac{1}{q^{19}} < \abs{p - q \pi} < \frac{8}{x^{1/4}}  \, .
\]
Solving for $q$ gives $q > 8^{-1/19} x^{1/76}$.
Hence 
\[
  p > q \pi - 1 > 8^{-1/19} \pi x^{1/76} - 1 > 2 x^{1/76} - 1 \ge x^{1/76} \ge k^{1/76}. 
\]
Thus, by the induction hypothesis, we have
\[
  W_{k+1} = W_k + p > W_k + k^{1/76} \ge \frac{1}{2} k^{77/76} + k^{1/76} .
\]
To finish, we use the inequality
\[
  (k + 1)^{77/76} - k^{77/76}
	  = \int_{k}^{k+1} \frac{77}{76} t^{1/76} \, dt
		\le \frac{77}{76} (k + 1)^{1/76}
		\le 2 k^{1/76} .
\]
Therefore, we have
\[
  W_{k + 1} 
	  > \frac{1}{2} k^{77/76} + k^{1/76}
		\ge \frac{1}{2} (k + 1)^{77/76} .   \qedhere
\]
\end{proof}

With our analysis of tame and wild terms out of the way, 
we can finally conclude that the infinite series converges.  

\begin{theorem}
The infinite series
\[
  \sum_{n=1}^\infty \frac{(\frac{2}{3} + \frac{1}{3} \sin n)^n}{n}
\]
converges.
\end{theorem}

\begin{proof}
We deal with the tame indexes and wild indexes separately. 
By Lemma~\ref{lemma:tame}, the tame sum is 
\[
  \sum_{n \text{ tame}} \frac{(\frac{2}{3} + \frac{1}{3} \sin n)^n}{n}
	  \le \sum_{n \text{ tame}} e^{- \sqrt{n}} 
		\le \sum_{n = 1}^{\infty} e^{- \sqrt{n}} \, ,
\]
which is bounded by $\sum_n \frac{5}{n^2}$ for example.  
By Lemma~\ref{lemma:wild}, the wild sum is 
\[
  \sum_{n \text{ wild}} \frac{(\frac{2}{3} + \frac{1}{3} \sin n)^n}{n}
	  \le \sum_{n \text{ wild}} \frac{1}{n}
		= \sum_{k = 1}^{\infty} \frac{1}{W_k}
		\le \sum_{k = 1}^{\infty} \frac{2}{k^{77/76}} \, ,
\]
which is bounded because the exponent~$77/76$ is bigger than~$1$.  
Because the tame and wild sums are bounded, 
the original series is bounded.  
Hence, because the series has only positive terms, it converges.
\end{proof}

\section{Open problem}

We now know that the infinite series has a finite value.
Our proof as is shows that the sum is less than~$200$.  
According to MathWorld~\cite{MathWorld}, 
the sum of the first $10$~million terms is approximately~$2.163$.
Is the sum of the infinite series a known constant?
We leave that question as an open problem.

\section*{Acknowledgment}

I am grateful to James Merryfield for pointing out that convergence of the series is related to rational approximations of~$\pi$.

\end{document}